\documentclass[a4paper]{amsart}

\usepackage[utf8]{inputenc}
\usepackage{lmodern}
\usepackage[T1]{fontenc}
\usepackage[british]{babel}
\usepackage{csquotes}
\usepackage{mathtools}
\usepackage{amssymb}
\usepackage{esint}
\usepackage{microtype}
\usepackage[shortcuts]{extdash}
\usepackage[style=alphabetic,maxnames=5,maxalphanames=4]{biblatex}
\usepackage{enumitem}
\usepackage{moreenum}
\usepackage{pgfplots}
\usepackage{accents}
\usepackage{hyperref}
\usepackage[capitalise]{cleveref}
\usepackage{crossreftools}

\hypersetup{
pdfauthor={Constantin Bilz and Julian Weigt},
colorlinks=true,
citecolor=[rgb]{0,0,0.5},
linkcolor=[rgb]{0,0,0.5},
urlcolor=[rgb]{0,0,0.75},
pdfstartview=FitH,
pdfpagemode=UseNone,
pdfdisplaydoctitle=true,
pdflang=en-GB,
breaklinks=true,
}

\newcommand{\R}{{\mathbb R}}
\newcommand{\Z}{{\mathbb Z}}
\newcommand{\I}{{\mathcal I}}
\newcommand{\Mu}{\mytilde{0.9}{M}}
\newcommand{\trf}{\mytilde{0.6}{f}}
\makeatletter
\newcommand{\mytilde}[2]{{\mathpalette{\m@tilde{#1}{#2}}\relax}}
\newcommand{\m@tilde}[4]{\accentset{\smash{\raisebox{-0.15ex}{\scalebox{#1}[0.6]{$#3\thicksim$}}}}#2}
\makeatother
\newcommand{\Msup}{{M_1}}
\newcommand{\Mnot}{{M_0}}
\newcommand\intd{{\,\mathrm{d}}}
\newcommand\D[1]{{\mathrm D#1}}

\makeatletter
\newcommand\@vsumbarthickness{0.035}
\newcommand\@vsumbarlength{1.12}
\DeclareMathOperator*\avsum{\mathpalette\@vsum\relax}
\newcommand\@vsum[2]{{
  \savebox\z@{$#1\sum$}
  \vphantom{\sum}
  \smash{\ooalign{%
    $#1\sum$\cr%
    \hidewidth$\vcenter{\hbox{\@vsumbar%
      {\@vsumbarthickness\wd\z@}%
      {\@vsumbarlength\wd\z@}%
    }}$\hidewidth\cr}}}}
\newcommand\@vsumbar[2]{%
  \begin{tikzpicture}
    \draw[line cap=round, line width=#1] (0,0) -- (0,#2);
  \end{tikzpicture}}
\makeatother

\DeclareMathOperator{\var}{var}

\newtheorem{lemma}{Lemma}[section]
\newtheorem{proposition}[lemma]{Proposition}
\newtheorem{theorem}[lemma]{Theorem}
\theoremstyle{definition}
\newtheorem{example}[lemma]{Example}
\theoremstyle{remark}
\newtheorem{remark}[lemma]{Remark}

\crefname{equation}{}{}
\crefname{subsection}{Section}{Sections}
\crefname{subsubsection}{Section}{Sections}
\crefname{enumi}{}{}

\numberwithin{equation}{section}

\pgfplotsset{compat=newest,
    width=0.8\textwidth,
    height=3.5cm,
    scale only axis=true,
    every axis/.style={
        axis line style={thick}
    },
    every axis plot/.append style={thick},
    tick style={black, thick},
    legend cell align={left},
    legend style={at={(axis cs:0,1)}, anchor=north, draw=white!30!black},
    tick align=outside,
    tick pos=left,
    typeset ticklabels with strut,
    stepfct style/.style={black},
    maxfct style/.style={very thick, black, dash pattern=on 4pt off 2.472135954999579pt]},
    maxfct style 2/.style={very thick, black, dash pattern=on 2pt off 1.2360679775pt},
}

\begin{document}

\title
[The centred maximal function diminishes the variation]
{The one-dimensional centred maximal function diminishes the variation of indicator functions}

\makeatletter
\hypersetup{pdftitle={\@title}}
\makeatother

\author[C. Bilz]{Constantin Bilz}
\address[CB]{School of Mathematics,
University of Birmingham,
Edgbaston,
Birmingham,
B15~2TT,
England}
\email{%
\href{mailto:c.bilz@pgr.bham.ac.uk}%
{c.bilz@pgr.bham.ac.uk}}
\urladdr{%
\href{https://sites.google.com/view/bilz}%
{https://sites.google.com/view/bilz}}

\author[J. Weigt]{Julian Weigt}
\address[JW]{Aalto University,
Department of Mathematics and Systems Analysis,
P.O.\linebreak Box~11100,
FI-00076 Aalto,
Finland}
\email{%
\href{mailto:julian.weigt@aalto.fi}
{julian.weigt@aalto.fi}}
\urladdr{%
\href{https://math.aalto.fi/~weigtj1/}
{https://math.aalto.fi/~weigtj1/}}

\subjclass[2020]{Primary 42B25; Secondary 26A45}
\keywords{Maximal function, bounded variation, indicator function}

\begin{abstract}
We prove sharp local and global variation bounds for the centred Hardy--Littlewood maximal functions of indicator functions in one dimension.
We characterise maximisers, treat both the continuous and discrete settings and extend our results to a larger class of functions.
\end{abstract}

\maketitle

\section{Introduction}\label{s:intro}

We are concerned with sharp variation bounds for the centred Hardy\--Littlewood maximal function $Mf$ on the real line $\R$ defined by
\[
M f(x)
=
\sup_{r > 0}
\fint_{x-r}^{x+r} |f(y)| \intd y
=
\sup_{r > 0}
\frac1{2r} \int_{x-r}^{x+r} |f(y)| \intd y.
\]
The variation of a function $f \colon \R \to \R$ on an interval $I \subseteq \R$ is
\[
\var_I(f)
=
\sup_{\text{$\phi \colon \Z \to I$ monotone}}
\sum_{i \in \Z} |f(\phi(i)) - f(\phi(i+1))|.
\]
We write $\var(f) = \var_\R(f)$ and say that \(f\) is of \emph{bounded variation} if $\var(f) < \infty$.

\Textcite{Kur15} proved that for any such function it holds that
\begin{equation}%
\label{e:var_bd}
\var(M f) \leq C \var(f)
\end{equation}
for some large constant $C$ independent of $f$.
It is an open conjecture that the optimal constant in this inequality is $C=1$, see e.g.~\cite{BCHP12,Kur15}.
The following main result proves this in the case of indicator functions.

\begin{theorem}%
\label{tind}
Let $f \colon \R \to \{0,1\}$ be a function of bounded variation.
Then \cref{e:var_bd} holds with $C=1$.
Equality is attained if and only if $f$ is constant or the set $\{x \in \R \mid f(x) = 1\}$ is a bounded interval of positive length.
\end{theorem}

Note that an indicator function is of bounded variation precisely if it has at most finitely many jumps.
This directly implies that $f(x) = 0$ or $f(x) = Mf(x)$ for Lebesgue\-/almost every $x \in \R$.
Our methods only require this weaker assumption, allowing us to prove the following more general result for nonnegative functions.

\begin{theorem}%
\label{t}
Let $f \colon \R \to [0,\infty)$ be a function of bounded variation such that for almost every $x \in \R$ we have that $f(x) = 0$ or $f(x) = Mf(x)$.
Then \cref{e:var_bd} holds with $C=1$.
Equality is attained if and only if \(f\) is constant or the set $\{x \in \R \mid f(x) > 0\}$ is a bounded interval of positive length and for any $x \in \R$, \[\liminf_{y \to x} f(y) \leq f(x) \leq \limsup_{y \to x} f(y).\]
\end{theorem}

The regularity of maximal functions was first studied by \textcite{Kin97} who proved that the $d$\=/dimensional centred Hardy\--Littlewood maximal operator is bounded on the Sobolev space $W^{1,p}(\R^d)$ when $1<p\leq\infty$ and $d \geq 1$.
\Textcite{HO04} later asked whether the endpoint inequality
\begin{equation}\label{e:HO04}
\|\nabla Mf\|_{L^1(\R^d)} \leq C \|\nabla f\|_{L^1(\R^d)}
\end{equation}
also holds and \citeauthor{Kur15}'s inequality~\cref{e:var_bd} provides a positive answer to this question in the one\-/dimensional case.
The higher\-/dimensional case remains completely open.

In comparison to the one\-/dimensional \emph{centred} Hardy\--Littlewood maximal function, its \emph{uncentred} counterpart
\[
\Mu f(x)
=
\sup_{x_0 < x < x_1} \frac1{x_1-x_0} \int_{x_0}^{x_1} |f(y)| \intd y
\]
allows averages over a larger class of intervals and hence may be expected to be smoother.
Indeed, \textcite{Tan02} gave a short proof of the uncentred version of \eqref{e:var_bd} with $C=2$ and later \textcite{AP07} showed that the optimal constant is $C = 1$.
\Textcite{Ram19} studied the sharp version of \cref{e:var_bd} for a family of \emph{nontangential} maximal functions interpolating between the centred and uncentred Hardy--Littlewood maximal functions.

Similarly, higher\-/dimensional partial results are available for the uncentred maximal function where the corresponding results are not known in the centred case.
The first such result is due to \textcite{MR2539555} who proved the uncentred version of \cref{e:HO04} for so\-/called block decreasing functions.
Later, \textcite{Lui18} proved the same for radial functions and the second author~\cite{Wei20} proved the corresponding inequality for indicator functions.

\subsection{Discrete setting}\label{s:introdiscrete}

Our methods also imply discrete analogues of \cref{tind,t}.
The discrete centred Hardy\--Littlewood maximal function $M f \colon \Z \to \R$ of a bounded function $f \colon \Z \to \R$ is defined by
\[
M f(n) =
\sup_{r \in \Z_{\geq 0}}
\avsum_{m=n-r}^{n+r} |f(m)|
=
\sup_{r \in \Z_{\geq 0}}
\frac1{2r+1}
\sum_{m=n-r}^{n+r} |f(m)|.
\]
For a \emph{discrete interval} $I \subseteq \Z$, i.e.\ the intersection of $\Z$ and a real interval, the variation of \(f\) on \(I\) is
\[
\var_I(f) = \sum_{n,n+1 \in I} |f(n) - f(n+1)|.
\]
We say that $f$ is of \emph{bounded variation} if $\var_\Z(f) < \infty$.

\Textcite{BCHP12} proved that
\[
\var_\Z(Mf) \leq C \sum_{n \in \Z} |f(n)|
\]
for $C = 2 + \frac{146}{315}$.
They asked whether the optimal constant in this inequality is $C=2$ and whether the stronger inequality
\begin{equation}\label{e:var_bdZ}
\var_\Z(M f) \leq C \var_\Z(f)
\end{equation}
analogous to \cref{e:var_bd} holds.
\Textcite{Mad17} affirmatively answered the first question and \textcite{Tem13} adapted \citeauthor{Kur15}'s method to prove \cref{e:var_bdZ} with a non\-/optimal constant.
We improve \citeauthor{Tem13}'s result by establishing the optimal constant $C=1$ in the case of indicator functions.

\begin{theorem}\label{tindZ}
Let $f \colon \Z \to \{0,1\}$ be a function of bounded variation.
Then \cref{e:var_bdZ} holds with $C=1$.
Equality is attained if and only if $f$ is constant or the set $\{n \in \Z \mid f(n) = 1\}$ is a bounded nonempty discrete interval.
\end{theorem}

In fact this result quickly follows from the continuous \cref{tind} and an embedding argument.
In the same way, we also establish the following relationship between the optimal constants in the continuous and discrete variation bounds for general functions of bounded variation.

\begin{proposition}\label{p:emb}
If \cref{e:var_bd} holds for all functions of bounded variation,
then the same is true for \cref{e:var_bdZ} with the same constant.
\end{proposition}

However, we do not know whether a similar embedding argument can be used to prove the following discrete analogue of the stronger \cref{t}.
This is mainly because of the additional assumptions in these theorems.
We circumvent this issue by adapting the proof of \cref{t} to the discrete setting.

\begin{theorem}%
\label{t:Z}
Let $f \colon \Z \to [0,\infty)$ be a function of bounded variation such that for any $n \in \Z$ we have $f(n) = 0$ or $f(n) = Mf(n)$.
Then \cref{e:var_bdZ} holds with $C=1$.
Equality is attained if and only if $f$ is constant or the set $\{n \in \Z \mid f(n) > 0\}$ is a bounded nonempty discrete interval.
\end{theorem}

Although the proofs of \cref{t,t:Z} are quite similar, different technical difficulties arise in each case.
In the continuous setting, we have to deal with compactness issues and exceptional sets of measure zero.
In the discrete setting, one inconvenience is that not every integer interval has an integer midpoint.

\subsection{Proof strategy}

Let us explain some ideas of the proofs using the example of the continuous setting.
Our main observation is that for a function $f\colon \R \to [0,\infty)$ satisfying the assumptions of \cref{t}, the local variation bound
\begin{equation}
\label{eq:localvarbound}
\var_{[a,b]}(Mf)
\leq
\var_{[a,b]}(f)
\end{equation}
holds for any real numbers $a < b$ such that $f(a)=Mf(a)$ and $f(b)=Mf(b)$, i.e.\ such that $Mf$ is \emph{attached} to $f$ at \(a\) and \(b\).
Our proof of \cref{t} heavily relies on this property.
The following example shows a typical situation.
Denote
\[
\chi_{[a,b]}(x)
=
\begin{cases*}
1 & if $a<x<b$, \\
1/2 & if $x=a$ or $x=b$, \\
0 & otherwise.
\end{cases*}
\]
\begin{example}%
\label{ex:two_bumps}
Let $c \in (1,3)$ and
$
f
=
\chi_{[-c,-1]} + \chi_{[1,c]}
$.
Then $Mf$ is attached to $f$ at any point $x$ with $1 \leq |x| \leq c$ and
\[
\var_{[-1,1]}(Mf)
=
c^{-1}
<
1
=
\var_{[-1,1]}(f).
\]
The maximal function \(M f\) has a strict local maximum of value $(c-1)/c$ at $0$ and two strict local minima of value $(3c-3)/(4c)$ at $\pm c/3$,
see \cref{fig:two_bumps}.
\end{example}

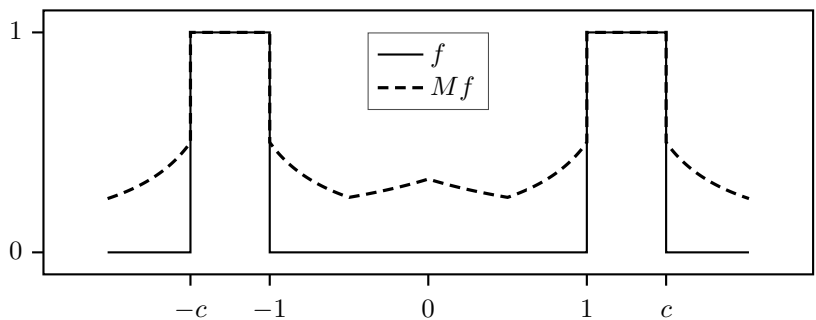
\begin{figure}[b]
\centering
\begin{tikzpicture}
\begin{axis}[
xtick={-3/2,-1,0,1,3/2},
xticklabels={$-c$,$-1$,$0$,$1$,$c$},
ytick={0,1},
]
\addplot [stepfct style]
table {data_twot_f.dat}; %generated by ../examples/gen_plot_data.py
\addlegendentry{$f$}
\addplot [maxfct style]
table {data_twot_Mf.dat}; %generated by ../examples/gen_plot_data.py
\addlegendentry{$Mf$}
\end{axis}
\end{tikzpicture}
\caption{%
The functions $f$ and $Mf$ in \cref{ex:two_bumps} with $c=3/2$.
}
\label{fig:two_bumps}
\end{figure}

The calculations leading to \cref{ex:two_bumps,fig:two_bumps}, as well as to \cref{ex:much_variation,fig:much_variation,fig:defp} below, are straightforward because for step functions it holds that
\[
Mf(x) = \sup_{\text{$y \neq x$ is a jump of $f$}} \fint_{x-|x-y|}^{x+|x-y|} |f(z)| \intd z.
\]

The local variation bound \cref{eq:localvarbound} will follow from part~\cref{i:abcd_bounded} of the following result.
An analogue for unbounded intervals is contained in part~\cref{i:abcd_unbounded}.

\begin{proposition}\label{l:abcd}%
Let \(f \colon \R \to [0,\infty)\) be a bounded Lebesgue\-/measurable function
and let \(I \subseteq \R\) be an interval
such that $f(x) = 0$ for almost every $x\in I$.
Then the following holds:
\begin{enumerate}[label=(\arabic*)]
\item\label{i:abcd_bounded}%
If $I = [a,b]$ for some real numbers $a<b$, then
$\var_{[a,(a+b)/2]}(Mf) \leq Mf(a)$ and $\var_{[(a+b)/2,b]}(Mf) \leq Mf(b)$.
Both of these inequalities are strict unless $f$ vanishes almost everywhere on $\R$.
\item\label{i:abcd_unbounded}%
If $I = (-\infty,a]$ or $I=[a,\infty)$ for some real $a$, then $Mf$ is monotone on $I$ and $\var_I(Mf) = Mf(a) - \inf_{x \in I} Mf(x)$.
\end{enumerate}
\end{proposition}

We prove this local variation bound in \cref{s:l} and we apply it in \cref{s:proof_t} to show \cref{t} and hence \cref{tind}.
In \cref{s:abcdZ} we prove an analogous discrete local variation bound which we then apply in \cref{s:proofZ,s:eqZ} to show the discrete \cref{t:Z}.
These proofs can be read mostly independently from \cref{s:l,s:proof_t}.
\Cref{s:emb} contains the embedding argument leading to \cref{p:emb} and the derivation of \cref{tindZ} from \cref{tind}.

Our approach can be compared to the strategy in \cite{AP07} for the uncentred Hardy\--Littlewood maximal function \(\Mu f\).
They show that if $f \colon \R \to \R$ is of bounded variation and satisfies $f(x) = \limsup_{y \to x} f(y)$ for any $x \in \R$, then \(\Mu f\geq f\) and \(\Mu f\) is attached to $f$ at any strict local maximum point of \(\Mu f\).
This can be used to show~\cref{eq:localvarbound} when $Mf$ is replaced by \(\Mu f\) and \(a\) and \(b\) are neighbouring strict local maximum points of \(\Mu f\).

However, in the centred case, $Mf$ is not necessarily attached to $f$ at strict local maxima of $Mf$,
see \cref{ex:two_bumps} above.
We overcome this obstruction by making use of a gradient bound for $Mf$ in the proof of \cref{l:abcd}.
On the other hand, this bound becomes less useful for our purposes if a function fails to satisfy the assumptions of \cref t.
In fact, for general functions of bounded variation, the local variation bound \cref{eq:localvarbound} can fail between some points of attachment.
This prevents us from generalising our results to a substantially larger class of functions than in \cref t.

\begin{example}%
\label{ex:much_variation}
Let $h=2/5$ and $f = \chi_{[-3/2,-1]} + h \cdot \chi_{[-1/2,1/2]} + \chi_{[1,3/2]}$.
Then $f$ is constant in $(-1/2,1/2)$ and $Mf$ is attached to $f$ at any point $x$ with $2 \leq 8|x| \leq 3$, but $Mf$ has a strict local maximum of value $7/15>h$ at~$0$.
In particular, \cref{eq:localvarbound} fails between the points of attachment $a=-1/3$ and $b=1/3$,
see \cref{fig:much_variation}.
\end{example}

\begin{figure}
\centering
\begin{tikzpicture}
\begin{axis}[
xtick={-1,-1/3,0,1/3,1},
xticklabels={$-1$,$a$,$0$,$b$,$1$},
ytick={0,1},
]
\addplot [stepfct style]
table {data_muchvar_f.dat}; %generated by ../examples/gen_plot_data.py
\addlegendentry{$f$}
\addplot [maxfct style]
table {data_muchvar_Mf.dat}; %generated by ../examples/gen_plot_data.py
\addlegendentry{$Mf$}
\end{axis}
\end{tikzpicture}
\caption{%
The functions $f$ and $Mf$ in \cref{ex:much_variation}.
}
\label{fig:much_variation}
\end{figure}
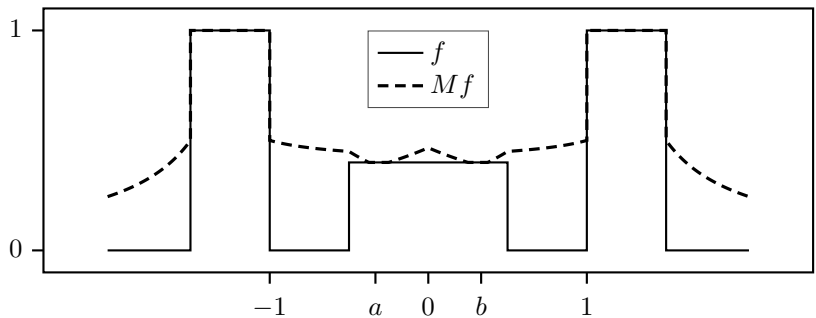

\subsection{Further remarks}

\subsubsection*{Maximisers and maximising sequences}

It follows from our results that maximisers for \cref{e:var_bd} exist in the class of indicator functions and in the larger class in \cref{t}.
However, not every maximising sequence converges pointwise modulo symmetries to a nonzero maximiser, e.g.\ take $c \to 1$ in \cref{ex:two_bumps}.

\subsubsection*{Sobolev variation}

Another common notion of variation is given by the total variation \(|\D f|(\R^d)\) of the distributional derivative \(\D f\), i.e.\ the measure satisfying the integration by parts rule
\[
\int_{\R^d} f\varphi'\intd x
=-
\int_{\R^d}\varphi\intd(\D f)
\]
for all functions \(\varphi\in C^\infty_{\mathrm c}(\R^d)\).
The variation of a function on \(\R^d\) with \(d>1\) is usually defined in this way.
For any function $f\colon\R\to\R$ of bounded variation it holds that \(|\D f|(\R)\leq\var(f)\).
Conversely, if \(|\D f|(\R)<\infty\), then
there exists a function \(\bar f\) equal to \(f\) almost everywhere such that \(\var(\bar f)=|\D f|(\R)\), see e.g.\ \cite[Theorem~7.2]{MR2527916}.
If $f$ satisfies the hypotheses of \cref t, then it follows that
\[
|\D{M f}|(\R)
\leq
\var(Mf)
=
\var(M\bar f)
\leq
\var(\bar f)
=
|\D f|(\R).
\]
Hence \cref{t} remains true for this definition of the variation.

\section*{Acknowledgements}

The first author was partially supported by the Deutsche Forschungsgemeinschaft under Germany's Excellence Strategy, Project 390685813 -- EXC-2047/1, during a stay at the Hausdorff Research Institute for Mathematics.
The second author has been partially supported by the Vilho, Yrjö and Kalle Väisälä Foundation of the Finnish Academy of Science and Letters,
and the Magnus Ehrnrooth foundation.

We would like to thank the supervisors of the first author, Diogo Oliveira e Silva and Jonathan Bennett, and the supervisor of the second author, Juha Kinnunen, for all their support.

\section{Proof of \cref{l:abcd}}
\label{s:l}

Throughout this section,
let $f \colon \R \to [0,\infty)$ be a bounded measurable function.
The following result proves the unbounded case in \cref{l:abcd}\cref{i:abcd_unbounded}.
By symmetry, it suffices to take \(I=[a,\infty)\).

\begin{lemma}%
\label{l:mong}
Let \(a\in\R\) be such that $f(x) = 0$ for almost every $x\geq a$.
Then \(Mf\) is nonincreasing on \([a,\infty)\)
and hence \[\var_{[a,\infty)}(Mf) = Mf(a) - \inf_{x \in [a,\infty)} Mf(x).\]
\end{lemma}
\begin{proof}
Let \(a\leq x\leq y\).
By the definition of $Mf$ and the assumptions on $f$,
\[
Mf(x)
=
\sup_{r>x-a}
\fint_{x-r}^{x+r} f(z) \intd z
\geq
\sup_{r>x-a}
\fint_{x-r}^{x+r+2(y-x)} f(z) \intd z
=
Mf(y)
.
\]
This completes the proof.
\end{proof}

The rest of this section is devoted to the proof of \cref{l:abcd}\cref{i:abcd_bounded}, i.e.\ the case that $I=[a,b]$ for some real numbers $a<b$.
It suffices to consider the special case that $a=-1$ and $b=1$ and to prove the strict inequality
\begin{equation}\label{e:abcd01}
\var_{[0,1]}(Mf) < Mf(1)
\end{equation}
under the assumption that $f(x) = 0$ for almost every $x \in [-1,1]$ and that $f$ does not vanish almost everywhere on $\R$.
The general case follows from this because for any nonconstant affine map $\phi \colon \R \to \R$ we have that $M(f \circ \phi)(1) = Mf(\phi(1))$ and
\[
\var_{\phi([0,1])} (Mf)
= \var_{[0,1]} ((Mf) \circ \phi)
= \var_{[0,1]} (M (f \circ \phi)).
\]

For the proof of \cref{e:abcd01}
we first note that $Mf$ restricted to $[0,\infty)$ is the pointwise maximum of the auxiliary maximal functions $\Mnot f, \Msup f \colon [0,\infty) \to [0,\infty)$ defined by
\begin{align*}
\Mnot f(x)
=
\sup_{r \leq 1+x}
\fint_{x-r}^{x+r} f(y) \intd y,
\qquad
\Msup f(x)
&=
\sup_{r \geq 1+x}
\fint_{x-r}^{x+r} f(y) \intd y,
\end{align*}
see \cref{fig:defp} for an example.
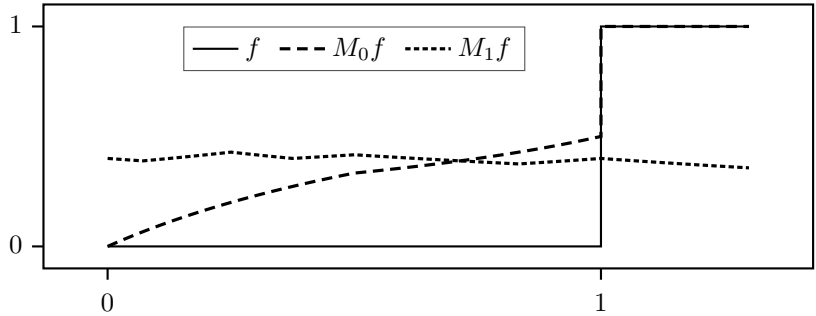
\begin{figure}[b]
\centering
\begin{tikzpicture}
\begin{axis}[
xtick={0, 1},
ytick={0, 1},
legend columns=3,
legend style={at={(axis cs:0.5,1)}},
]
\addplot [stepfct style]
table {data_defp_f.dat}; %generated by ../examples/gen_plot_data.py
\addlegendentry{$f$\;\,}
\addplot [maxfct style]
table {data_defp_Mnotf.dat}; %generated by ../examples/gen_plot_data.py
\addlegendentry{$\Mnot f$\;\,}
\addplot [maxfct style 2]
table {data_defp_Msupf.dat}; %generated by ../examples/gen_plot_data.py
\addlegendentry{$\Msup f$}
\end{axis}
\end{tikzpicture}
\caption{The auxiliary maximal functions $\Mnot f$ and $\Msup f$ on $[0,1]$ when $f = \chi_{[-5/2, -2]} + \chi_{[-3/2, -1]} + \chi_{[1, 2]} + \chi_{[3, 7/2]}$.}
\label{fig:defp}
\end{figure}%
Of these, $\Msup f$ only permits averages over large radii.
Based on this, our first lemma bounds the difference quotients of $\Msup f$.

\begin{lemma}%
\label{l:gradbd}
Let $x,y \geq 0$ be distinct and let $r \geq 1+x$ be such that \[\Msup f(x) = \sup_{s \geq r} \fint_{x-s}^{x+s} f(z)\intd z.\]
Then,
\[
\frac{\Msup f(x) - \Msup f(y)}{|x-y|}
\leq
\frac{\Msup f(x)}{r+|x-y|}
\leq
\frac{\Msup f(y)}r.
\]
\end{lemma}

Note that by the definition of $\Msup f$, we can always take \(r\) to be at least \(1+x\).
The lemma also holds for \(Mf\) instead of $\Msup f$, but then we are not guaranteed a good lower bound on $r$.

\begin{proof}
We have that $\Msup f(x) < \infty$ since $f$ is bounded.
Hence, for any $\epsilon > 0$ there exists an $s \geq r$ such that $(1-\epsilon) \Msup f(x) \leq \fint_{x-s}^{x+s} f(z) \intd z$ and therefore,
\begin{align*}
(1-\epsilon)\Msup f(x) - \Msup f(y)
&\leq
\fint_{x-s}^{x+s} f(z) \intd z
%- \fint_u^v f(z) \intd z \\
- \fint_{y-s-|x-y|}^{y+s+|x-y|} f(z) \intd z \\
&\leq
\Bigl(
\frac1{2s}-\frac1{2s + 2|x-y|}
\Bigr)
\int_{x-s}^{x+s} f(z) \intd z
\\
&=
\frac{|x-y|}{s+|x-y|} \fint_{x-s}^{x+s} f(z) \intd z \\
&\leq
\frac{|x-y|}{r+|x-y|} \Msup f(x).
\end{align*}
The first inequality uses the definition of $\Msup f(y)$ together with the fact that $s+|x-y| \geq 1+y$.
In the second inequality, we use the nonnegativity of $f$ to reduce the domain of integration of the second integral.
The last two relations follow from definitions.
Now the first inequality in the lemma follows by letting \(\epsilon\to0\).
The second inequality follows after rearranging terms.
\end{proof}

Bounds similar to \cref{l:gradbd} have frequently appeared in the literature, including in higher dimensions.
The related inequality \(|\nabla M_\alpha f(x)|\leq C M_{\alpha-1} f(x)\)
for the fractional maximal function \(M_\alpha f\) with \(1\leq\alpha\leq d\) was proved in \cite{KS03}.
A generalisation to the range \(0\leq\alpha\leq d\) can be found in \cite[Section~2.5]{beltran2021continuity}.

We now employ the previous result to prove a local variation bound for $\Msup f$.
The strictness of this inequality will be crucial to the characterisation of maximisers in our results.

\begin{lemma}%
\label{l:varbd}
It holds that $\var_{[0,1]}(\Msup f) \leq \Msup f(1)$
and this inequality is strict
if \(f(x)=0\) for almost every $x \in [-1,1]$ and $f(x) > 0$ for any $x$ in some set of positive measure.
\end{lemma}
\begin{proof}
First assume that \(f(x)=0\) for almost every $x \in [-1,1]$ and $f(x) > 0$ for any $x$ in some set of positive measure.
Then $\Msup f(x) > 0$ for any $x \geq 0$.
By \cref{l:gradbd}, $\Msup f$ is continuous.
Since the map $(x,s) \mapsto \fint_{x-s}^{x+s} f(y) \intd y$ is continuous at $(x,s) = (0,1)$ and $f(y)=0$ for almost every \(y\in[-1,1]\), this implies the existence of a \(\delta \in (0,1)\) such that for any \(x \in [0, \delta)\),
\[
\Msup f(x)
=
\sup_{s\geq 1+x+\delta}
\fint_{x-s}^{x+s}f(y)\intd y
\]
and hence $f$ and $x$ satisfy the hypotheses of \cref{l:gradbd} with $r = 1+x+\delta$.
Without the additional assumptions that \(f(x)=0\) for almost every \(x\in[-1,1]\) and that $f(x)>0$ for any $x$ in some set of positive measure, this remains true for \(\delta=0\).

In order to estimate the variation of $\Msup f$ on $[0,1]$, we let $\epsilon > 0$, $k \geq 1$ and \[0=x_0<x_1<\ldots<x_k=1\]
be such that $x_{i+1} - x_i < \epsilon$ for $0 \leq i \leq k-1$.
We write $\delta_i = \delta$ if $x_{i+1} < \delta$ and $\delta_i = 0$ otherwise.
By the two inequalities in \cref{l:gradbd},
\begin{align*}
\sum_{i=0}^{k-1}
|\Msup f(x_i) - \Msup f(x_{i+1})|
&\leq
\sum_{i=0}^{k-1}
\frac{x_{i+1}-x_i}{1+x_i+\delta_i}
\max(\Msup f(x_i),\Msup f(x_{i+1}))
\\
&\leq
\sum_{i=0}^{k-1}
\frac{(2+\delta_i)(x_{i+1}-x_i)}{(1+x_i+\delta_i)^2}
\Msup f(1).
\end{align*}
By viewing the last sum as a Riemann sum, taking the supremum over all possible choices of $k$ and $x_i$ as above and letting $\epsilon \to 0$ we obtain that
\begin{align*}
\var_{[0,1]}(\Msup f)
&\leq
\biggl(
\int_0^{\delta}
\frac{2+\delta}{(1+x+\delta)^2}
\intd x
+
\int_{\delta}^1
\frac2{(1+x)^2}
\intd x
\biggr)
\Msup f(1)
\\
&\leq
\int_0^1
\frac2{(1+x)^2}
\intd x
\cdot
\Msup f(1)
\\
&=
\Msup f(1)
\end{align*}
and the second inequality is strict if \(\delta>0\).
This completes the proof.
\end{proof}

\begin{remark}
\label{r:abcd_sobolevproof}
Let us sketch a shorter but less elementary version of the second part of the above proof.
By \cref{l:gradbd}, the auxiliary maximal function \(\Msup f\) is Lipschitz continuous.
Hence it is differentiable almost everywhere and
\[\var_{[0,1]}(\Msup f)=\int_0^1|(\Msup f)'(x)|\intd x.\]
At any point of differentiability $x \in (0,1)$ we have by \cref{l:gradbd} that
\[|(\Msup f)'(x)|\leq\frac{\Msup f(x)}{1+x}\leq\frac{2\Msup f(1)}{(1+x)^2}\]
and the first inequality is strict in some neighbourhood of $0$.
Plugging this into the above formula for \(\var_{[0,1]}(\Msup f)\) yields \cref{l:varbd}.
\end{remark}

The next lemma concerns the other auxiliary maximal function $\Mnot f$.

\begin{lemma}%
\label{l:mon}
Let $f(x) = 0$ for almost every $x \in [-1,1]$.
Then $\Mnot f$ is nondecreasing on $[0,1]$.
\end{lemma}
\begin{proof}
This is similar to the proof of \cref{l:mong}.
Let \(0 < x \leq y \leq 1\).
Then,
\[
\Mnot f(x)
=
\sup_{1-x<r\leq1+x}
\fint_{x-r}^{x+r} f(z) \intd z
\leq
\sup_{1-x<r\leq1+x}
\fint_{x-r+2(y-x)}^{x+r} f(z) \intd z
\leq
\Mnot f(y)
.
\]
Since $\Mnot f(0)=0$ and $\Mnot f$ is nonnegative, this completes the proof.
\end{proof}

We have established the monotonicity of $\Mnot f$ in \cref{l:mon} and a variation bound for $\Msup f$ in \cref{l:varbd}.
The next result will allow us to deduce a variation bound for the pointwise maximum $M f = \max(\Mnot f, \Msup f)$.

\begin{lemma}%
\label{l:varmax}
Let $g,h \colon [0,1] \to \R$ be functions such that $g(1) \leq h(1)$ and let $g$ be nondecreasing.
Then $\var_{[0,1]}(\max(g,h)) \leq \var_{[0,1]}(h)$.
\end{lemma}
\begin{proof}
Write $u = \max(g,h)$.
We need to show that for any $k \geq 1$ and any \[0 = x_0 < x_1 < \ldots < x_k = 1\] it holds that
\[
\sum_{i=0}^{k-1}
|u(x_i)-u(x_{i+1})|
\leq \var_{[0,1]}(h).
\]
Write $x_{k+1} = 1$ and let $p(-1) < p(0) < \ldots < p(\ell)$ be the elements of the set
\[
P = \{-1\} \cup \{i \in \{0,\ldots,k\} \mid h(x_i) \geq u(x_{i+1})\}.
\]
Clearly, $p(-1)=-1$.
Since $x_k=x_{k+1}=1$ and by assumption, $h(x_k) = u(x_{k+1})$ and hence $p(\ell)=k$.
If $i \in \{0,\ldots,k\} \setminus P$, then $h(x_i) < u(x_{i+1})$.
Since $g$ is nondecreasing, we also have that $g(x_i) \leq g(x_{i+1})$ and hence $u(x_i) \leq u(x_{i+1})$.
On the other hand, if $i \in P \setminus \{-1\}$, then
\[
h(x_i) \geq u(x_{i+1}) \geq g(x_{i+1}) \geq g(x_i).
\]
Hence, $h(x_i) = u(x_i)$ and $u(x_i)\geq u(x_{i+1})$.
This shows that for any $0 \leq j \leq \ell$,
\[
u(x_{p(j-1)+1})
\leq
u(x_{p(j-1)+2})
\leq
\ldots
\leq
u(x_{p(j)})
=
h(x_{p(j)})
\geq
u(x_{p(j)+1})
.
\]
We conclude that
\begin{align*}
\sum_{i=0}^{k-1}
|u(x_i)-u(x_{i+1})|
&=
\sum_{j=0}^\ell
\sum_{i=p(j-1)+1}^{p(j)}
|u(x_i)-u(x_{i+1})|
\\
&=
\sum_{j=0}^\ell
2h(x_{p(j)}) - u(x_{p(j-1)+1}) - u(x_{p(j)+1})
\\
&\leq
\sum_{j=0}^\ell
2h(x_{p(j)}) - h(x_{p(j-1)+1}) - h(x_{p(j)+1})
\\
&\leq
\var_{[0,1]}(h).
\end{align*}
This completes the proof.
\end{proof}

We are now ready to prove \cref{e:abcd01}.
Let $f(x) = 0$ for almost every $x \in [-1,1]$ and
let $h \colon [0,1] \to [0,\infty)$ be the function defined by \(h(x)=\Msup f(x)\) for \(0\leq x<1\) and \(h(1)=Mf(1)\).
Then \(\Mnot f(1) \leq h(1)\) and $Mf$ restricted to $[0,1]$ is the pointwise maximum of $\Mnot f$ and $h$.
Hence by an application of \cref{l:mon,l:varmax} and then \cref{l:varbd},
\[
\var_{[0,1]}(Mf)
\leq
\var_{[0,1]} (h)
\leq
\var_{[0,1]}(\Msup f)
+
Mf(1)-\Msup f(1)
\leq
M f(1).
\]
The last inequality is strict if $f$ does not vanish almost everywhere on $\R$.
This shows \cref{e:abcd01} and hence completes the proof of \cref{l:abcd}.

\begin{remark}
One can show that $\Mnot f$ is also Lipschitz continuous on $[0,1]$ and that \(\Mnot f\) and \(\Msup f\) do not coincide at more than one point in $[0,1]$ if $f$ does not vanish almost everywhere on $\R$.
Let us only sketch a proof of the fact that if $y \in [0,1]$ is such that $\Mnot f(y) \geq \Msup f(y)$, then $\Mnot f(x) > \Msup f(x)$ for any $x \in (y,1]$.
By \cref{l:gradbd},
\[
\frac{\Msup f(x) - \Msup f(y)}{x-y}
\leq
\frac{\Msup f(y)}{1+x}.
\]
Similarly as in the proofs of \cref{l:gradbd,l:mon}, one can show that
\[
\frac{\Mnot f(x) - \Mnot f(y)}{x-y}
\geq
\frac{\Mnot f(y)}{1-x+2y}.
\]
Since $\Mnot f(y) \geq \Msup f(y) > 0$ and $y<x$ it follows that $\Mnot f(x) > \Msup f(x)$.
\end{remark}

\section{Proof of \cref{t}}\label{s:proof_t}

Throughout this section, let $f \colon \R \to [0,\infty)$ be a function of bounded variation such that for almost every $x \in \R$ we have that $f(x) = 0$ or $f(x) = Mf(x)$.
In order to prove \cref{t}, we need to show the inequality
\begin{equation}
\label{e:var1}
\var(Mf) \leq \var(f)
\end{equation}
and determine its cases of equality.
We will accomplish this by using a certain canonical representative $\bar f$ whose properties facilitate the application of \cref{l:abcd}.
In \cref{s:repr}, we define $\bar f$, show that $f$ and $\bar f$ agree almost everywhere and that
\begin{equation}
\label{e:varfbarf}
\var(\bar f)\leq\var(f).
\end{equation}
There, we also establish some further properties of $\bar f$.
In \cref{s:appl_abcd}, we apply \cref{l:abcd} to show \cref{e:var1} for $\bar f$, i.e.\ we show that
\begin{equation}
\label{e:var1bar}
\var(Mf) \leq \var(\bar f)
.
\end{equation}
Inequalities \cref{e:varfbarf,e:var1bar} together imply \cref{e:var1}.
In \cref{s:equality}, we characterise the cases of equality in \cref{e:var1}
by characterising the cases of equality in \cref{e:var1bar}
and then characterising the cases of equality in \cref{e:varfbarf} under the assumption that equality holds in \cref{e:var1bar}.

\subsection{Canonical representative}\label{s:repr}

Let us define a function $\bar f \colon \R \to [0,\infty)$ as follows.
If $x \in \R$ is such that
\begin{equation}\label{eq:pointaverage}
\limsup_{r\searrow0}
\fint_{x-r}^{x+r}
f(y)\intd y
= 0,
\end{equation}
then we let $\bar f(x)=0$ and otherwise we let $\bar f(x) = Mf(x)$.
This \emph{canonical representative} is related to but distinct from the homonymous object in \cite{AP07}.
By the Lebesgue differentiation theorem and the assumption on $f$, we have that $f(x) = \bar f(x)$ for almost every $x \in \R$ and hence $Mf(x) = M\bar f(x)$ for any $x \in \R$.
Since $f$ is of bounded variation, its one\-/sided limits exist at any point.
It follows that \cref{eq:pointaverage} can be rewritten without the use of an integral, but we will not need this.

The following lemma will be used multiple times throughout this section.

\begin{lemma}\label{l:semicont}
The maximal function $Mf$ is lower semi\-/continuous, i.e.\ for any $x \in \R$ it holds that
$
\liminf_{y \to x} Mf(y) \geq Mf(x)
$.
\end{lemma}
\begin{proof}
By definition, $Mf$ is the pointwise supremum of the continuous functions
\[
\R \ni x \mapsto \fint_{x-r}^{x+r} f(y) \intd y, \quad r > 0.
\]
The lemma follows from this.
\end{proof}

We now show that the canonical representative does not increase the variation.

\begin{lemma}\label{l:varbarf}%
Inequality \cref{e:varfbarf} holds.
\end{lemma}
\begin{proof}
We first claim that it suffices to show that for any $x \in \R$ and $\epsilon > 0$ there exist $y_1,y_2 \in (x-\epsilon,x+\epsilon)$ such that
$
f(y_1) - \epsilon \leq \bar f(x) \leq f(y_2) + \epsilon
$.
Let $k \geq 1$ and let
\[
-\infty < x_0 < x_1 < \ldots < x_k < \infty
.
\]
By iteratively removing any points $x_i$ with $1 \leq i \leq k-1$ for which $\bar{f}(x_i)$ lies in the convex hull of $\{\bar{f}(x_{i-1}), \bar{f}(x_{i+1})\}$,
we obtain a subsequence \(x_0'<\ldots<x_\ell'\) such that
\[
\sum_{i=0}^{k-1} |\bar f(x_i) - \bar f(x_{i+1})|
=
\sigma
\sum_{i=0}^{\ell-1} (-1)^i \bar f(x'_i) + (-1)^{i+1} \bar f(x'_{i+1})
\]
for some $\sigma\in\{-1,1\}$.
Let $\epsilon > 0$.
By assumption, there exist points $y_i \in (x'_i-\epsilon,x'_i+\epsilon)$ such that
\[
\sigma (-1)^i \bar f(x'_i) \leq \sigma (-1)^i f(y_i) + \epsilon
\]
for any $0 \leq i \leq \ell$.
If $\epsilon$ is small enough, then $y_i$ is increasing in $i$ and hence
\[
\sum_{i=0}^{k-1} |\bar f(x_i) - \bar f(x_{i+1})|
- 2\ell\epsilon
\leq
\sigma
\sum_{i=0}^{\ell-1} (-1)^i f(y_i) + (-1)^{i+1} f(y_{i+1})
\leq
\var(f).
\]
Let $\epsilon \to 0$ and then take the supremum over all $k$ and $x_i$ as above to show \cref{e:varfbarf}.

It remains to show that for any \(x\in\R\) and $\epsilon > 0$ there exist points $y_1$ and $y_2$ as above.
Let $r \in (0,\epsilon)$.
We start with the existence of $y_1$.
By the definitions of $\bar f(x)$ and $Mf(x)$, \[\bar f(x) \geq \limsup_{r \searrow 0} \fint_{x-r}^{x+r} f(y) \intd y.\]
Hence if $r$ is sufficiently small, then the integral on the right\-/hand side is at most $\bar f(x) + \epsilon$ and so there exists a $y_1 \in (x-r,x+r)$ with $f(y_1) - \epsilon \leq \bar f(x)$, as required.

We complete the proof by showing the existence of $y_2$.
If $\bar f(x) = 0$, then we may simply choose $y_2=x$ because $f$ is nonnegative.
So we assume that $\bar f(x) = Mf(x) > 0$.
Since \cref{eq:pointaverage} fails, $f(y) > 0$ for any $y$ in some subset of positive measure of $(x-r,x+r)$.
As $f$ and $\bar f$ are equal almost everywhere, it follows that $f(y_2) = \bar f(y_2) = Mf(y_2)$ for some $y_2 \in (x-r,x+r)$.
Hence if $r$ is small enough, then \cref{l:semicont} implies that $f(y_2)+\epsilon \geq \bar f(x)$, as required.
\end{proof}

In particular, \cref{e:varfbarf} shows that $\bar f$ is of bounded variation.
Together with the definition of $\bar f$, this implies some topological properties of the vanishing set
\[V=\{x \in \R \mid \bar f(x)=0\}.\]

\begin{lemma}%
\label{l:onesided0}
The set $V$ is open and its boundary has no limit points in $\R$.
\end{lemma}
This can be stated equivalently as follows:
There exists a finite or countably infinite nondecreasing sequence of points \(a_i \in \R\cup\{\pm\infty\}\) without accumulation points in \(\R\)
such that \(V=\bigcup_i(a_{2i},a_{2i+1})\).
\begin{proof}
If $f$ vanishes almost everywhere, then $V = \R$ and the lemma follows.
Since $f$ is nonnegative, we may therefore assume that $f$ is positive in a set of positive measure.
Let \(x\in\R\).
Then $Mf(x) > 0$ and by \cref{l:semicont} there exists an \(\epsilon>0\) such that \(Mf(y)>\epsilon\) for any \(y\in(x-\epsilon,x+\epsilon)\).

We first show that $x$ is not a limit point of the boundary of $V$.
Let $k \geq 0$ and let
\[x-\epsilon<x_0<x_1<\ldots<x_{2k+1}<x+\epsilon\] be a sequence of points with \(\bar{f}(x_{2i})=0\) and \(\bar{f}(x_{2i+1})=M\bar{f}(x_{2i+1})\) for any $0 \leq i \leq k$.
It suffices to show that $k$ is bounded by a constant that only depends on $\bar{f}$ and $\epsilon$.
Such a bound holds because
\[
(2k+1)\epsilon
<
\sum_{i=0}^{2k}|\bar{f}(x_{i+1})-\bar{f}(x_i)|
\leq
\var(\bar{f})
<\infty
.
\]
The first inequality above holds by the properties of $\epsilon$ and $x_k$.
The second inequality holds by definition.
Hence $x$ is not a limit point of the boundary of $V$.

It remains to show that $V$ is open.
To this end, let $x$ be a boundary point of $V$.
We need to show that $\bar f(x) > 0$.
By the first part of the proof, $f(y) > \epsilon$ for any $y$ in some one\-/sided neighbourhood of $x$, i.e.\ for any $y$ in $(x-r,x)$ or $(x,x+r)$ for some $r>0$.
Since $f$ is nonnegative, we see that \cref{eq:pointaverage} fails and hence $\bar f(x) = Mf(x) > 0$.
This completes the proof.
\end{proof}

\subsection{Global variation bound}\label{s:appl_abcd}

In \cref{s:repr}, we proved \cref{e:varfbarf}.
Together with the following result, this implies \cref{e:var1}, proving the first part of \cref{t}.

\begin{proposition}%
\label{l:intermediate}%
Inequality \cref{e:var1bar} holds.
\end{proposition}

\begin{proof}
By \cref{l:onesided0} and a subdivision of $\R$ we see that \cref{e:var1bar} holds if
\begin{equation}\label{e:varIbar}
\var_I(Mf) \leq \var_I(\bar f)
\end{equation}
whenever \(I\) is a connected component of $\R \setminus V$ or the closure of a connected component of $V$.
If \(I\) is a connected component of \(\R\setminus V\), then \(\bar f\) and \(Mf\) agree on \(I\),
so that \cref{e:varIbar} holds with equality.
Now let \(I\) be the closure of a connected component of \(V\).
If \(I=\R\), then both sides of \cref{e:varIbar} are zero.
On the other hand, if \(I\neq\R\),
then by \cref{l:onesided0}, $\bar f$ and $Mf$ agree on the boundary of $I$
and therefore \cref{e:varIbar} follows from either \cref{i:abcd_bounded} or \cref{i:abcd_unbounded} in \cref{l:abcd}.
This completes the proof.
\end{proof}

\subsection{Cases of equality}\label{s:equality}

It remains to characterise the cases of equality in \cref{e:var1}.
We first establish certain regularity properties of \(\bar f\).

\begin{lemma}
\label{l:onesided0positive}
Any connected component of $V$ or $\R \setminus V$ has positive length.
\end{lemma}
\begin{proof}
Let \(x\in\R\).
If $\bar f(y) > 0$ for any $y \neq x$ in some compact neighbourhood of $x$, then by \cref{l:semicont} there exists an $\epsilon > 0$ such that $\bar f(y) = Mf(y) > \epsilon$ for any such $y$.
Hence \cref{eq:pointaverage} fails and $\bar f(x) = Mf(x) > \epsilon$.
This shows that $\{x\}$ is not a connected component of $V$.

On the other hand, if $\bar f(y) = 0$ for any $y \neq x$ in some neighbourhood of $x$, then \cref{eq:pointaverage} holds and hence $\bar f(x) = 0$.
This shows that $\{x\}$ is not a connected component of $\R \setminus V$.
Since $x \in \R$ was arbitrary, it follows that any connected component of $V$ or $\R \setminus V$ has positive length.
\end{proof}

Now we investigate the behaviour of the canonical representative \(\bar f\) on connected components of its support \[\R \setminus V = \{x \in \R \mid \bar f(x) > 0\}.\]
This set is closed by \cref{l:onesided0}.
Our next result will only be applied in the case of an unbounded connected component, but its proof is identical in the bounded and unbounded cases.

\begin{lemma}
\label{l:concave}
The function $\bar f$ is concave on any connected component of $\R \setminus V$.
\end{lemma}
\begin{proof}
Suppose for a contradiction that \(\bar f\) is not concave on some connected component \(I\) of $\R \setminus V$.
Then there exist points \(x_0 < x_1 < x_2\) in $I$ such that $\bar f(x_1) < L(x_1)$ where $L \colon \R \to \R$ is the affine linear function defined by $L(x_0)=\bar f(x_0)$ and $L(x_2)=\bar f(x_2)$.
Hence for $g = \bar f - L$ we have $g(x_1) < 0$ and $g(x_0)=g(x_2)=0$.

Since $\bar{f}$ and $Mf$ are equal in $I$, \cref{l:semicont} and the continuity of $L$ imply that there exists a smallest $x_1' \in [x_0,x_2]$ such that
\[
g(x_1') = \inf_{x_0 \leq y \leq x_2} g(y) < 0.
\]
Since $g(x_0) = 0$, there exists an $r > 0$ such that $[x_1'-r,x_1'+r] \subseteq [x_0,x_2]$.
We have that $g(y) \geq g(x_1')$ for any $y \in [x_0,x_2]$ and the inequality is strict if $y < x_1'$.
Hence by the mean value property for $L$,
\begin{align*}
Mf(x_1')
\geq
\fint_{x_1'-r}^{x_1'+r} \bar{f}(y) \intd y
=
\fint_{x_1'-r}^{x_1'+r} g(y) \intd y
+ L(x_1')
>
g(x_1') + L(x_1')
=
\bar{f}(x_1').
\end{align*}
This is a contradiction to the fact that $x_1' \in \R \setminus V$.
Therefore \(\bar{f}\) is concave on \(I\).
\end{proof}

The following result is a consequence of \cref{l:concave}.

\begin{lemma}
\label{l:concave2}
Let $I$ be an unbounded connected component of $\R \setminus V$.
Then, \[\lim_{|x| \to \infty;\, x\in I} \bar f(x) > 0.\]
Furthermore, if $I=\R$, then $\bar f$ is constant.
\end{lemma}

\begin{proof}
Suppose for a contradiction that one of the conclusions of the lemma is false.
Let $x_0 \in I$, meaning that $\bar f(x_0) > 0$.
Then by symmetry, we may assume that $[x_0,\infty) \subseteq I$ and that
there exists a $x_1 > x_0$ such that $\bar f(x_1) < \bar f(x_0)$.
By \cref{l:concave}, it follows that $\bar f(x_2) \leq L(x_2)$ for any $x_2 \geq x_1$ where $L \colon \R \to \R$ is the affine linear function defined by $L(x_0) = \bar f(x_0)$ and $L(x_1) = \bar f(x_1)$.
Notice that $L$ is strictly decreasing and hence $\bar f(x_2)<0$ if $x_2$ is large enough.
This is a contradiction to the nonnegativity of $\bar f$.
\end{proof}

We can now characterise the cases of equality in the intermediate inequality \cref{e:var1bar}.

\begin{proposition}%
\label{l:intermediateequality}%
Equality holds in \cref{e:var1bar} if and only if $\bar f$ is constant or $\R\setminus V$ is a compact interval of positive length.
\end{proposition}
\begin{proof}
It suffices to consider the case that $\bar f$ is not constant since otherwise both sides of \cref{e:var1bar} are zero.
Then $\R \setminus V$ is nonempty.
By the second part of \cref{l:concave2}, we also have that $V$ is nonempty.

By the proof of \cref{l:intermediate},
equality holds in \cref{e:var1bar} if and only if \cref{e:varIbar}
holds with equality whenever \(I\) is the closure of some connected component of $V$.
Any such $I$ has positive length by \cref{l:onesided0positive}.
By the strictness in \cref{l:abcd}\cref{i:abcd_bounded}, this means that \cref{e:var1bar} can only hold with equality if all connected components of \(V\) are unbounded, i.e.\ if $\R \setminus V$ is a nonempty interval.
This interval is closed by \cref{l:onesided0} and
has positive length by \cref{l:onesided0positive}.

Now let $I\neq\R$ be an unbounded connected component of $V$.
Since the function $\bar f$ is of bounded variation, its limits at $\pm\infty$ exist and for any $x \in \R$,
\[
Mf(x)
\geq
\lim_{r \to \infty} \fint_{x-r}^{x+r} \bar f(y) \intd y
=
\lim_{y \to \infty} \frac{\bar f(y) + \bar f(-y)}2.
\]
Furthermore, if the right\-/hand side is zero, then $\lim_{|x| \to \infty} Mf(x) = 0$.
By \cref{l:abcd}\cref{i:abcd_unbounded}, it follows that \cref{e:varIbar} holds with equality if and only if $\lim_{|x| \to \infty} \bar f(x) = 0$.
By \cref{l:concave2}, this is the case precisely when $\R \setminus V$ has no unbounded components.
We conclude that if $\bar f$ is not constant, then \cref{e:var1bar} holds with equality if and only if $\R \setminus V$ is a compact interval of positive length.
This completes the proof.
\end{proof}

We can now characterise the cases of equality in \cref{e:var1}.
If $\bar f$ is constant, then $Mf$ is constant.
In this case, equality in \cref{e:var1} holds precisely when $f$ is also constant.
We may now assume that $\bar f$ is not constant.
In this case, by \cref{e:varfbarf}, \cref{e:var1bar} and \cref{l:intermediateequality}, equality holds in \cref{e:var1} if and only if \cref{e:varfbarf} holds with equality and $\R \setminus V=[a,b]$ for some real numbers $a<b$.
If $\R \setminus V$ is of this form, then the canonical representative \(\bar f\) is concave on \([a,b]\) by \cref{l:concave}.
Hence it is continuous on \((a,b)\) and
\[
0 \leq \bar f(a) \leq \lim_{x \searrow a} \bar f(x)
\quad \text{and} \quad
0 \leq \bar f(b) \leq \lim_{x \nearrow b} \bar f(x)
.
\]
Furthermore, $\bar f$ vanishes on $\R \setminus [a,b]$.
Since $f$ and $\bar f$ are equal almost everywhere, this implies that \cref{e:varfbarf} holds with equality if and only if \(f(x)=\bar f(x)\) for any \(x \in \R \setminus \{a,b\}\) and
\[
0 \leq f(a) \leq \lim_{x \searrow a} f(x)
\quad \text{and} \quad
0 \leq f(b) \leq \lim_{x \nearrow b} f(x)
.
\]
This completes the proof of \cref{t}.

\section{Discrete setting}\label{s:discrete}

In this section, we first use an embedding argument to prove the conditional result \cref{p:emb} and to derive the discrete \cref{tindZ} from the continuous \cref{tind}.
Afterwards, we adapt the arguments in \cref{s:l,s:proof_t} to show the more general discrete \cref{t:Z}.

\subsection{Embedding}\label{s:emb}

Let $f \colon \Z \to \R$ be a function of bounded variation and let $Mf \colon \Z \to \R$ be the discrete maximal function as defined in \cref{s:introdiscrete}.
We define an associated step function $f_c \colon \R \to \R$ by
setting $f_c(x)=f(n)$ for any integer $n$ and any $x \in [n-1/2,n+1/2)$.
Let $Mf_c \colon \R \to \R$ be the continuous maximal function as defined in \cref{s:intro}.

For any monotone map $\phi \colon \Z \to \Z$ there exists a monotone map $\psi \colon \Z \to \R$ such that $f \circ \phi = f_c \circ \psi$ and vice versa.
Hence $\var_\Z(f) = \var(f_c)$.
Our next claim is that $\var_\Z(M f) \leq \var(M f_c)$.
This is an immediate consequence of the following result.

\begin{lemma}\label{l:emb}%
We have that $Mf(n) = Mf_c(n)$ for any integer $n$.
\end{lemma}
\begin{proof}
For any nonnegative integer \(m\),
the step function $f_c$ is constant on the intervals $[n-m-1/2,n-m+1/2)$ and $[n+m-1/2,n+m+1/2)$.
Thus for any positive radius $r$ with $|r-m| \leq 1/2$ we have that
\[
\frac1{2r}\int_{n-r}^{n+r} f_c(y) \intd y
=
\frac1{2r} \int_{n-m}^{n+m} f_c(y) \intd y
+ \frac{r - m}{2r}(f_c(n-m) + f_c(n+m)).
\]
The right\-/hand side is of the form \(A+B/r\) for some constants \(A\) and \(B\) independent of $r$,
where $B=0$ if $m=0$.
It follows that the map $r \mapsto \fint_{n-r}^{n+r} f_c(y) \intd y$ is constant on $(0,1/2]$ and monotone on $[m-1/2,m+1/2]$ for any positive integer $m$.
Hence,
\[
Mf_c(n)
=
\sup_{r \in \Z_{\geq 0}}
\fint_{n-r-1/2}^{n+r+1/2} f_c(y) \intd y
=
\sup_{r \in \Z_{\geq 0}}
\avsum_{m=n-r}^{n+r} f(m)
=
Mf(n).
\]
This completes the proof.
\end{proof}

If $f_c$ satisfies \cref{e:var_bd} for some constant $C$, then it follows from the above that
\begin{equation}
\label{e:varMZf}
\var_\Z(M f) \leq \var(Mf_c) \leq C \var(f_c) = C \var_\Z(f)
\end{equation}
and hence $f$ satisfies \cref{e:var_bdZ} with the same constant.
This proves \cref{p:emb} and enables us to derive \cref{tindZ} from \cref{tind}.

\begin{proof}[Proof of \cref{tindZ}]
By assumption, $f$ is $\{0,1\}$\=/valued and of bounded variation and so the same is true for $f_c$.
Hence by \cref{e:varMZf,tind}, we see that $f$ satisfies \cref{e:var_bdZ} with $C=1$.
Equality can only hold if equality holds in \cref{tind}.
For a nonconstant $f$, this implies that the set $\{x \in \R \mid f_c(x) = 1\}$ is a bounded interval of positive length
and hence the set $\{n \in \Z \mid f(n) = 1\}$ is a bounded nonempty discrete interval.
On the other hand, if $f$ is of this form, then equality is attained because for any integer $n$ with $f(n)=1$ we have that
\[
\var_\Z(Mf) \geq 2Mf(n) - \lim_{m \to \infty} Mf(m)+Mf(-m) = 2 - 0 = \var_\Z(f).
\]
This completes the proof.
\end{proof}

\subsection{Discrete local variation bound}\label{s:abcdZ}

The following result is the discrete analogue of \cref{l:abcd}.
We will use it to derive \cref{t:Z} similarly as \cref{t} in the continuous setting, but without any of the technical difficulties related to compactness issues or exceptional sets of measure zero.

\begin{proposition}%
\label{l:abcdZ}
Let \(f \colon \Z \to [0,\infty)\) be a bounded function and
let \(I \subseteq \R\) be an interval
such that $f(n) = 0$ for any integer $n$ in the interior of $I$.
Then the following holds:
\begin{enumerate}[label=(\arabic*)]
\item\label{i:abcd_boundedZ}%
If $I = [a,b]$ for some integers $a<b$, then
$\var_{I \cap \Z}(Mf) \leq Mf(a)+Mf(b)$.
The inequality is strict unless $f$ vanishes everywhere on $\Z$.
\item\label{i:abcd_unboundedZ}%
If $I = (-\infty,a]$ or $I=[a,\infty)$ for some integer $a$, then $Mf$ is monotone on $I \cap \Z$ and $\var_{I \cap \Z}(Mf) = Mf(a) - \inf_{n \in I\cap\Z} Mf(n)$.
\end{enumerate}
\end{proposition}
The proof of this result goes along similar lines of the proof of \cref{l:abcd},
although some details differ.
In particular we have to work around the fact that not all integer intervals have integer midpoints.

We first prove the unbounded case in \cref{l:abcdZ}\cref{i:abcd_unboundedZ}.
By symmetry, it suffices to take $I=[a,\infty)$.
\begin{lemma}%
\label{l:mongZ}
Let $f \colon \Z \to [0,\infty)$ be a bounded function and let $a \in \Z$ be
such that $f(n) = 0$ for every integer $n>a$.
Then \(Mf\) is nonincreasing on \([a,\infty)\cap\Z\)
and hence \[\var_{[a,\infty) \cap \Z}(Mf) = Mf(a) - \inf_{n \in [a,\infty) \cap \Z} Mf(n).\]
\end{lemma}
\begin{proof}
This is similar to the proof of \cref{l:mong}.
Let \(n,m\in\Z\) be such that \(a\leq n\leq m\).
Then,
\[
Mf(n)
=
\sup_{r \geq n-a}
\avsum_{k=n-r}^{n+r} f(k)
\geq
\sup_{r \geq n-a}
\avsum_{k=n-r}^{n+r+2(m-n)} f(k)
=
Mf(m)
.
\]
This completes the proof.
\end{proof}

The rest of this section is devoted to the proof of \cref{l:abcdZ}\cref{i:abcd_boundedZ},
i.e.\ the case that \(I=[a,b]\) for some integers $a<b$.
We start with a reduction using translation invariance.
We also insert a midpoint in the case that $a+b$ is odd.
For this, let \(f\) be as in \cref{l:abcdZ}.
Set
\[
S=
\begin{cases}
\Z&\text{if }a+b\text{ is even},\\
\Z+\tfrac12=\bigl\{\ldots,-\tfrac32,-\tfrac12,\tfrac12,\tfrac32,\ldots\bigr\}&\text{if }a+b\text{ is odd}
\end{cases}
\]
and write $S_0 = S \cup \{0\}$.
We define a translated function \(\trf\colon S\to[0,\infty)\) by \[\trf(n)=f\Bigl(n+\frac{a+b}2\Bigr)\] and we define its centred maximal function $M\trf \colon S_0 \to [0,\infty)$ by
\[
M\trf(n)
=
\sup_{v \in S;\,v \leq n}
\avsum_{m=v}^{2n-v} \trf(m)
.
\]
Given a domain \(T\in\{S,S_0\}\), a function \(g:T\to[0,\infty)\) and an interval \(I\subseteq\R\)
we define the variation of \(g\) on the \emph{discrete interval} $I\cap T$ by
\[
\var_{I\cap T}(g)
=
\sup_{\text{$\phi \colon \Z \to I\cap T$ monotone}}
\sum_{i \in \Z} |g(\phi(i)) - g(\phi(i+1))|.
\]
If $S=\Z$, then these definitions agree with those in \cref{s:introdiscrete}.
Note that
\begin{equation*}
\var_{[a,b]\cap\Z}(Mf)
=
\var_{[-(b-a)/2,(b-a)/2]\cap S}(M\trf)
\leq
\var_{[-(b-a)/2,(b-a)/2]\cap S_0}(M\trf)
\end{equation*}
and
\begin{equation*}
M\trf\Bigl(-\frac{b-a}2\Bigr)
=
Mf(a),
\qquad
M\trf\Bigl(\frac{b-a}2\Bigr)
=
Mf(b)
.
\end{equation*}

From now on and for the rest of the proof of \cref{l:abcdZ}\cref{i:abcd_boundedZ},
let $f \colon S \to [0,\infty)$ be a bounded nonzero function.
By the above relations and by symmetry, it is enough to show the strict inequality
\begin{equation}
\label{e:abcd01Z}
\var_{[0,a] \cap S_0}(Mf)
<
Mf(a)
\end{equation}
for any positive \(a\in S\) such that \(f(n)=0\) for all \(n\in S\) with \(-a<n<a\).
This is analogous to \cref{e:abcd01}.

Similarly as in the continuous setting, $Mf$ restricted to $[0,a] \cap S_0$ is the pointwise maximum of the auxiliary maximal functions $\Mnot f, \Msup f \colon [0,a] \cap S_0 \to [0,\infty)$ defined by
\[
\Mnot f(n)
=
\max_{v \in S;\,-a<v\leq n}
\avsum_{m=v}^{2n-v} f(m)
,
\qquad
\Msup f(n)
=
\sup_{v \in S;\,v \leq -a}
\avsum_{m=v}^{2n-v} f(m)
.
\]

The following gradient bound for $\Msup f$ is analogous to the continuous \cref{l:gradbd}.
Since admissible radii in the above discrete setting are separated by a distance of~$1$, an additional term $1/2$ appears in this bound.
Because of this, we also dispense with the additional lower bound on the radii in \cref{l:gradbd}.
Except for these differences, the proof is similar to the continuous case.

\begin{lemma}%
\label{l:gradbdZ}
Let $n,m \in [0,\infty) \cap S_0$ be distinct.
Then,
\[
\frac{\Msup f(n) - \Msup f(m)}{|n-m|}
\leq
\frac{\Msup f(n)}{n+a+1/2+|n-m|}
\leq
\frac{\Msup f(m)}{n+a+1/2}
.
\]
\end{lemma}

\begin{proof}
We have $\Msup f(n) < \infty$ since $f$ is bounded.
Hence for any $\epsilon > 0$ there exists a $v \in S$ with $v \leq -a$ such that $(1-\epsilon) \Msup f(n) \leq \avsum_{k=v}^{2n-v} f(k)$.
Let \(w\) be such that \[m-w = n-v+|n-m|.\]
Then $w \in S$ because \(v-w\) is an integer.
Since \(w\leq v < 2n-v\leq2m-w\),
\begin{align*}
\nonumber
(1-\epsilon)\Msup f(n) - \Msup f(m)
&\leq
\avsum_{k=v}^{2n-v} f(k)
-
\avsum_{k=w}^{2m-w} f(k)
\\
\nonumber
&\leq
\Bigl(
\frac1{2(n-v)+1}
-
\frac1{2(m-w)+1}
\Bigr)
\sum_{k=v}^{2n-v} f(k)
\\
\nonumber
&=
\frac{2|n-m|}{2(m-w)+1}
\avsum_{k=v}^{2n-v} f(k)
\\
&\leq
\frac{|n-m|}{n+a+1/2+|n-m|}
\Msup f(n).
\end{align*}
The first, third and fourth relations follow from definitions and the fact that $w \leq -a$.
In the second line, we use that $f$ is nonnegative to reduce the range of summation of the second sum.
Now the first inequality in the lemma follows by letting \(\epsilon\to0\).
The second inequality follows after rearranging terms.
\end{proof}

Our next result is a local variation bound for $\Msup f$ analogous to the continuous \cref{l:varbd}.
Here the proof is somewhat simplified due to a telescoping argument.
Furthermore, due to the additional term $1/2$ in \cref{l:gradbdZ} above, we are able to show a slightly stronger inequality than in the continuous setting.
This artefact already allows us to obtain
a strict inequality,
whereas in the continuous setting we have to work a little harder to get the strict inequality in \cref{l:varbd}.

\begin{lemma}%
\label{l:varbdZ}
Let \(a\in S\) be nonnegative.
Then,
\[
\var_{[0,a] \cap S_0}(\Msup f) \leq \frac{2a}{2a+1} \Msup f(a)
.
\]
\end{lemma}
\begin{proof}
Let $n < m$ be elements of $[0,a] \cap S_0$.
By the two inequalities in \cref{l:gradbdZ},
\begin{align*}
|\Msup f(n) - \Msup f(m)|
&\leq
\frac
{m-n}
{m+a+1/2}
\max(\Msup f(n),\Msup f(m))
\\
&\leq
\frac
{(m-n)(2a+1/2)}
{(n+a+1/2)(m+a+1/2)}
\Msup f(a)
\\
&=
\Bigl(
\frac
{2a+1/2}
{n+a+1/2}
-
\frac
{2a+1/2}
{m+a+1/2}
\Bigr)
\Msup f(a)
.
\end{align*}
Now let
$
0 = n_0 < n_1 < \ldots < n_k = a
$
be an enumeration of $[0,a] \cap S_0$.
We use the above estimate
and evaluate the resulting telescoping sum
to obtain that
\begin{align*}
\var_{[0,a] \cap S_0}(\Msup f)
&=
\sum_{i=0}^{k-1}
|\Msup f(n_i) - \Msup f(n_{i+1})|
\leq
\Bigl(
\frac
{2a+1/2}
{a+1/2}
-
1
\Bigr)
\Msup f(a)
.
\end{align*}
This completes the proof.
\end{proof}

Regarding the other auxiliary maximal function $\Mnot f$, the following result similar to \cref{l:mon,l:mongZ} holds.

\begin{lemma}\label{l:monZ}
Let $a \in S$ be nonnegative and let $f(n)=0$ for any $n \in S$ with $-a<n<a$.
Then $\Mnot f$ is nondecreasing on $[0,a] \cap S_0$.
\end{lemma}
\begin{proof}
Let $n,m \in S$ be such that $0 < n \leq m \leq a$.
Then,
\[
\Mnot f(n)
=
\max_{v\in S;\, -a < v \leq 2n-a}
\avsum_{k=v}^{2n-v} f(k)
\leq
\max_{v \in S;\, -a < v \leq 2n-a}
\avsum_{k=v+2(m-n)}^{2n-v} f(k)
\leq
\Mnot f(m)
.
\]
Since $\Mnot f(0)=0$ and $\Mnot f$ is nonnegative, this completes the proof.
\end{proof}
Having established the monotonicity of $\Mnot f$ and a variation bound for $\Msup f$ similarly as in the continuous setting, the next step is to combine these results using the following analogue of \cref{l:varmax}.
We omit the proof because it is the same.

\begin{lemma}\label{l:varmaxZ}
Let $a \in S$ be nonnegative.
Let $g,h \colon [0,a] \cap S_0 \to \R$ be functions such that $g(a) \leq h(a)$ and let $g$ be nondecreasing.
Then, \[\var_{[0,a] \cap S_0}(\max(g,h)) \leq \var_{[0,a] \cap S_0}(h).\]
\end{lemma}

We are now ready to prove \cref{e:abcd01Z}.
Let $a \in S$ be positive such that $f(n) = 0$ for any $n \in S$ with $-a < n < a$ and
let $h \colon [0,a] \cap S_0 \to [0,\infty)$ be the function defined by $h(n) = \Msup f(n)$ for $n < a$ and $h(a) = Mf(a)$.
Then $\Mnot f(a) \leq h(a)$ and $Mf$ restricted to $[0,a] \cap S_0$ is the pointwise maximum of $\Mnot f$ and $h$.
Hence we can apply \cref{l:monZ,l:varmaxZ} and then \cref{l:varbdZ} to obtain that
\begin{align*}
\var_{[0,a] \cap S_0} (Mf)
&\leq
\var_{[0,a] \cap S_0} (h)
\leq
\var_{[0,a] \cap S_0} (\Msup f)
+
Mf(a)
-
\Msup f(a)
<
Mf(a)
.
\end{align*}
This proves \cref{e:abcd01Z} and thus completes the proof of \cref{l:abcdZ}.

\subsection{Discrete global variation bound}\label{s:proofZ}

We now prove the inequality in \cref{t:Z}.
Throughout this section and the next section,
let $f \colon \Z \to [0,\infty)$ be a function of bounded variation
such that for any $n \in \Z$ we have $f(n) = 0$ or $f(n) = Mf(n)$.
For possibly infinite endpoints $a \leq b$ we write
\[[a,b]\cap\Z=\{n\in\Z\mid a\leq n\leq b\}.\]

There exists a possibly unbounded discrete interval $\mathcal I \subseteq \Z$ with at least two elements and a nondecreasing sequence $(a_i)_{i \in \I}$ of points in $\Z \cup \{\pm \infty\}$ such that
\begin{equation*}
\{n\in\Z\mid f(n)>0\}
=
\bigcup_{i,i+1\in\I ;\, i \text{ odd}}
[a_i, a_{i+1}] \cap \Z
=
\Z \setminus
\bigcup_{i,i+1\in\I ;\, i \text{ even}}
(a_i, a_{i+1})
\end{equation*}
and \(a_i+2\leq a_{i+1}\) for any even \(i\in\I\) such that \(i+1\in\I\).
We may further assume that the points $\pm \infty$ each occur at most once in the sequence $(a_i)_{i \in \I}$.

Let $i \in \I$ be such that $i+1 \in \I$.
If $i$ is even, then by \cref{l:abcdZ},
\begin{equation}
\label{e:varIbarZ}
\var_{[a_i,a_{i+1}]\cap\Z}(Mf)
\leq
\var_{[a_i,a_{i+1}]\cap\Z}(f).
\end{equation}
On the other hand, if $i$ is odd, then by assumption it holds that \(f(n)=Mf(n)\) for all \(n\in[a_i,a_{i+1}]\cap\Z\) and thus \cref{e:varIbarZ} holds with equality.
We can conclude that
\begin{equation}
\label{e:var1barZ}
\var_\Z(Mf)
=
\sum_{i,i+1\in\I}\var_{[a_i,a_{i+1}]\cap\Z}(Mf)
\leq
\sum_{i,i+1\in\I}\var_{[a_i,a_{i+1}]\cap\Z}(f)
=
\var_\Z(f)
.
\end{equation}
This proves the inequality in \cref{t:Z}.

\subsection{Cases of equality}\label{s:eqZ}

For the characterisation of the cases of equality in \cref{t:Z},
we may assume that $f$ is not constant since otherwise both sides of \cref{e:var1barZ} are zero.
By the last subsection, equality holds in \cref{e:var1barZ} if and only if for every even \(i\in\I\) with $i+1\in\I$ we have equality in \cref{e:varIbarZ}.
We need the following concavity result whose proof we omit because it is similar to the proof of \cref{l:concave2}.
The conclusion of this result slightly differs from \cref{l:concave2} because here we already assume $f$ to be nonconstant.
\begin{lemma}
\label{l:concave2Z}
Let \(i\in\I\) be odd and such that \(i+1\in\I\) and \(a_{i+1}=\infty\).
Then $\lim_{n \to \infty} f(n) > 0$ and $a_i > -\infty$.
\end{lemma}

Since $f$ is not constant, it is not the zero function.
Hence $\I$ is not of the form $\{i,i+1\}$ for any even $i$.
By \cref{l:concave2Z}, it is also not of this form for any odd $i$.
Hence $\I$ has at least three elements.
If there exists an even \(i\in\I\) with \(i+1\in\I\) and \(a_i,a_{i+1}\in\Z\),
then \cref{e:varIbarZ} is a strict inequality by \cref{l:abcdZ}
and hence \cref{e:var1barZ} is strict.
It remains to consider the case that no such \(i\) exists.
After re\-/indexing and up to symmetry, this means that \(\I\) is either \(\{0,1,2,3\}\) or $\{0,1,2\}$.

In the first case, $f$ is finitely supported
and hence, by \cref{l:abcdZ}\cref{i:abcd_unboundedZ},
equality holds in \cref{e:varIbarZ} for the even indices $i=0$ and $i=2$.
Thus \cref{e:var1barZ} holds with equality.
In the second case, by \cref{l:concave2Z},
\[
Mf(n)
\geq
\lim_{r \to \infty} \avsum_{m=n-r}^{n+r} f(m)
=
\lim_{m \to \infty} \frac{f(m)}2
> 0.
\]
for any integer $n$.
By \cref{l:abcdZ}\cref{i:abcd_unboundedZ}, this means that \cref{e:varIbarZ} is strict for $i=0$ and hence \cref{e:var1barZ} is strict.
We conclude that equality holds in \cref{e:var1barZ} if and only if $f$ is constant or $\{n \in \Z \mid f(n)>0\} = [a,b] \cap \Z$ for some integers $a \leq b$.
This completes the proof of \cref{t:Z}.

\printbibliography

\end{document}